\documentclass[11pt]{article}
\usepackage[utf8]{inputenc}
\usepackage{amssymb}
\usepackage{amsmath,dsfont}
\usepackage{multirow}
\usepackage{xcolor}
\usepackage[inner=2cm,outer=2cm]{geometry}
\usepackage[enableskew]{youngtab}

\newtheorem{theorem}{Theorem}[section]
\newtheorem{lemma}[theorem]{Lemma}
\newtheorem{proposition}[theorem]{Proposition}

\newtheorem{definition}[theorem]{Definition}

\newenvironment{proof}{\paragraph{Proof:}}{\hfill$\square$\\}

\usepackage{xparse}

\newcommand{\tib}[1]{{\color{purple} #1}}
\newcommand{\tr}{\operatorname{tr}}

\title{A coupling of the spectral measures at a vertex}
\author{Paul Rochet, Thibault Espinasse}
\begin{document}

% \buildtab{3}{2}{1}
% \buildtabcom{432}

\maketitle
\begin{abstract} 
Given the adjacency matrix of an undirected graph, we define a coupling of the spectral measures at the vertices, whose moments count the rooted closed paths in the graph. The resulting joint spectral measure verifies numerous interesting properties that allow to recover minors of analytical functions of the adjacency matrix from its generalized moments. We prove an extension of Obata's Central Limit Theorem in growing star-graphs to the multivariate case and discuss some combinatorial properties using Viennot's heaps of pieces point of view.
\end{abstract}


\section{Introduction}

A powerful tool to study the spectral properties of the adjacency matrix $A$ of a graph $G$ is given by the concept of spectral measure at a vertex, which encapsulates some information on the graph as seen from a particular vertex. This probability measure is fundamental in the study of the asymptotic spectrum distribution of random matrices \cite{charles1,charles2} and is related to the Benjamini-Schramm convergence of random rooted graphs \cite{BS}. In quantum mechanics, it is associated to quantum measurements with the adjacency matrix of the graph as the observable \cite{jacobs,obata2017spectral}.\\

% has been widely used in different settings, \cite{obata}, \cite{charles2} under different names, and is closely linked to Benjamini-Schramm convergence when dealing with a sequence of random rooted graphs \cite{BS}. As pointed in \cite{obata}, this measure have a physical interpretation in the context of quantum mechanics, in terms of result of a measurement. \\

The spectral measure at a vertex allows to define a random variable with values in the spectrum of $A$ and whose moments have simple combinatorial interpretations in term of closed paths. While it can be defined for every vertex of the graph, no multivariate extension have been proposed in the literature to our knowledge. This could be explained by the current lack of interpretation of a joint distribution in quantum mechanics, where Positive Operator Valued Measures \cite{povm} prevail as the most natural way to deal with the spectral measures simultaneously. In this paper, we address the question of looking for a natural coupling of these random variables. We introduce a joint spectral distribution in the form of a quasi-probability, that is, a signed measure with total mass $1$. This results in a quasi-random permutation $X$ of the spectrum of the adjacency matrix $A$ which verifies numerous interesting properties. In particular, the minors of $A^k, k \in \mathbb N$ can be expressed as generalized moments of $X$. Moreover, if the graph has no self-loop, the covariance matrix under the joint spectral measure is the Laplacian of the graph. We show this formalism to be compatible with the interpretation of the wave function of a multi-fermionic system as Slater determinants \cite{slater}, which appear naturally when calculating multivariate marginal distributions of the joint spectral measure, up to permutations. \\

Limit theorems on the spectral measures have been investigated for several types of growing graphs such as star graphs \cite{obatatcl} or comb graphs \cite{obata}. We extend Obata's result on star graphs to the multivariate case by considering the joint spectral measure in a graph obtained by merging a subset of vertices in $n$ copies of a graph, and letting $n$ tend to infinity. We show that, as in the univariate case, the limit only depends on the immediate neighborhood of the subset of vertices in the remaining part of the graph.\\

We explore some combinatorial properties that emerge from the joint distribution of a subset of the quasi-random variables in term of paths on the graph. Precisely, we show that the moment generating function of the quasi-random variables enumerates particular objects which can be described in terms of excursions on the graph, using Viennot's heaps of pieces point of view \cite{viennot1989heaps}.\\

% 
% This joint distribution defined in this paper relies on a determinantal structure. We investigate the combinatorial properties, and prove the multivariate CLT for star graphs, which extend previous work from Obata. We also enlights the combinatorial link with a particular type of paths on the graph which we call excursions.\\
% 
% The joint spectral measure encapsulates the spectral information of submatrices of analytical transformation of 

%
%
The paper is organized as follows. In the first section, we introduce the joint distribution and derive some properties related to its generalized moments. We prove the multivariate extension of Obata's Central Limit Theorem on star graphs in Section \ref{sec:3}. Then, we discuss a combinatorial aspect in relation with heaps of cycles and excursions on the graph. Finally, technical proofs of some of the results are presented in the Appendix.

\section{The joint spectral measure}\label{sec:2}
Let $A =\big( a_{ij} \big)_{i,j=1,...,N}$ be a symmetric matrix with real coefficients. The local spectral measure of $A$ at $i \in \{1,...,N\}$ is defined as the unique real measure $\mu_i$ with moments 
$$ \int_\mathbb R x^k d \mu_i(x) = \big(A^k\big)_{ii} \ , \ k=0,1,2,... $$
Recall that, because $A$ is symmetric, its eigenvalues $\lambda_1,...,\lambda_N$ are real and there exists an eigendecomposition $A = P \Lambda P^\top$ with $\Lambda= \operatorname{Diag}(\lambda_1,...,\lambda_n)$ a diagonal matrix and $P = (p_{ij})_{i,j=1,...,n}$ an orthogonal matrix. Without loss of generality, we will always assume $\det(P)=1$. One verifies easily that $\mu_i$ is the probability measure
$$ \mu_i = \sum_{k=1}^n p_{ik}^2 \delta_{\lambda_k},  $$
where $\delta_{\lambda_k}$ is the Dirac measure at $\lambda_k$. This measure has a combinatorial interpretation when $A$ is the adjacency matrix of a graph, in which case the $k$-th moment of $\mu_i$ counts the number \tib{of} walks of length $k$ from $i$ to itself on the graph. In this context, $\mu_i$ is called the spectral measure at root (or vertex) $i$. This can be generalized to any symmetric matrix viewing $A$ as the weighted adjacency matrix of a graph (possibly the complete graph). Then, the $k$-th moment of $\mu_i$ is equal to the sum of the weights of all walks of length $k$ from $i$ to itself, where the weight of a walk is the product of the entries of $A$ over the edges that compose it.\\

In this paper, we define a joint spectral measure on $\mathbb R^N$ whose marginal distributions are the univariate spectral measures $\mu_1,...,\mu_N$. Let $S_N$ be the set of permutations on $\{1,...,N \}$ and $\pi$ the signed measure on $S_N$ with density (with respect to the counting measure)
$$ \pi (\sigma) =  \epsilon(\sigma) \prod_{j=1}^N p_{j \sigma(j)} \ , \ \sigma \in S_N \ , $$ 
where $\epsilon(.)$ denotes the signature. Alternatively, $\pi (\sigma) = \det( P \odot M_{\sigma} )$ where $M_\sigma$ is the permutation matrix associated to $\sigma$ and $\odot$ is the Hadamard product. Remark that $\pi$ is a quasi-probability distribution, i.e.~a signed-measure with total mass $1$, in view of 
$$\sum_{\sigma \in S_N} \pi(\sigma) = \sum_{\sigma \in S_N} \epsilon(\sigma) \prod_{j=1}^N p_{j \sigma(j)}  = \det(P) = 1.  $$
Throughout the paper, $\lambda = (\lambda_1,...,\lambda_N)$ designates the vector of eigenvalues of $A$ (which may contain multiple occurrences). For $\sigma \in S_N$, we define $\lambda_\sigma := \big( \lambda_{\sigma(1)},..., \lambda_{\sigma(N)} \big)$ the associated vector of permuted eigenvalues. 

\begin{definition} The joint spectral measure of $A$ is the push-forward measure of $\pi$ by the application $\sigma \mapsto \lambda_\sigma$. In particular, $\mu$ has support $\{ \lambda_\sigma : \sigma \in S_N \}$ and density
$$ \mu \big(\lambda_\sigma \big) = \sum_{\tau: \lambda_{\tau} = \lambda_{\sigma}} \epsilon(\tau) \prod_{j=1}^N p_{j \tau(j)} \ , \ \sigma \in S_N .  $$
\end{definition}

If all eigenvalues $\lambda_i$ are distinct, then simply $\mu \big(\lambda_\sigma \big) = \pi (\sigma) = \epsilon(\sigma) \prod_{j=1}^N p_{j \sigma(j)}$. The definition is more general to account for multiple eigenvalues. \\

Although the definition involves the transformation matrix $P$, the joint spectral measure $\mu$ actually does not depend on the choice of the eigendecomposition. Precisely, if $A$ has some multiple eigenvalues and $A = P\Lambda P^\top = Q \Lambda Q^\top$ are two eigendecompositions with $\det(P) = \det(Q) = 1$, then 
$$ \forall \sigma \in S_N \ , \ \sum_{\tau: \lambda_{\tau} = \lambda_{\sigma}} \epsilon(\tau) \prod_{j=1}^N p_{j \tau(j)} = \sum_{\tau: \lambda_{\tau} = \lambda_{\sigma}} \epsilon(\tau) \prod_{j=1}^N q_{j \tau(j)}.  $$
For a direct combinatorial proof of this statement see Section \ref{a:basis} in the Appendix. Alternatively, it suffices to show that the generalized moments: 
$$ \forall k_1,...,k_N \in \mathbb N \ , \ m[k_1,...,k_N] := \int_{\mathbb R^N} x_1^{k_1}...x_N^{k_N} \ d\mu(x_1,...,x_N)$$
do not depend on the eigendecomposition ($\mu$ is characterized by its generalized moments as a signed measure on $\mathbb R^N$ with finite support). In fact, the generalized moments $m[k_1,...,k_N]$ have a rather simple expression, as we show in Lemma \ref{lem:gen_mom} below. Let $A[k_1,...,k_N]$ denote the matrix whose $i$-th column is the $i$-th column of $A^{k_i}$ for $i=1,...,N$ (with the convention $A^0 = I$, the identity matrix).

\begin{lemma}\label{lem:gen_mom} For all $k_1,...,k_N \geq 0$, $ m[k_1,...,k_N] =\det\big(A[k_1,...,k_N] \big)$.
\end{lemma}

\begin{proof} By definition, 
$$  m[k_1,...,k_N]  = \sum_{\sigma \in S_N} \lambda_{\sigma(1)}^{k_1}...\lambda_{\sigma(N)}^{k_N} \mu \big(\lambda_\sigma \big) = \sum_{\sigma \in S_N} \epsilon(\sigma) \prod_{i = 1}^N p_{i\sigma(i)} \lambda_{\sigma(i)}^{k_i}. $$
Since $A^k P = P \Lambda^k$ for all $k=0,1,...$, it follows that $ \big( A^k P \big)_{ij} = p_{ij} \lambda_j^k$ for all $i,j =1,...,N$ and $k\geq 0. $
Hence,
$$ m[k_1,...,k_N]  = \sum_{\sigma \in S_N} \epsilon(\sigma) \prod_{i = 1}^N \big( A^{k_i} P \big)_{i \sigma(i)}
=\det \big( A[k_1,...,k_N] P\big) = \det \big(A[k_1,...,k_N] \big). \vspace{-1.25cm}$$
\end{proof}
\vspace{0.3cm}

For ease of readability, we shall still use the standard notations for probability measures such as the probability $\mathbb P(.)$ of an event or the expectation $\mathbb E(.)$ defined similarly by integrating against $\mu$. To this aim, we introduce a quasi-random vector $X=(X_1,...,X_n)$ with distribution $\mu$, denoting e.g. $\mathbb E \big( f(X) \big) = \int_{\mathbb R^N} f(x) d \mu(x)$ for any integrable function $f : \mathbb R^N \to \mathbb R$. Quasi-random variables are common in quantum probability where they have a physical interpretation, typically for Wigner phase-space representations \cite{quasi}. 

%$$   $$
%
%We now prove that the marginal distributions of $\mu$ are the rooted spectral measure $\mu_1,...,\mu_N$. Although these can be obtained by direct calculation, a simple proof follows from deriving a general expression of the generalized moments of $\mu$. For $N$ integers $k_1,...,k_N \geq 0$, define
%$$ m[k_1,...,k_N] := \mathbb E\big( X_1^{k_1} ... X_N^{k_N} \big) =  \int_{\mathbb R^N} x_1^{k_1}...x_N^{k_N} \ d\mu(x_1,...,x_N) $$
%
%
%\noindent The marginal distributions of $\mu$ can now be deduced directly.
%, noticing that since $\mu$is characterized by its generalized moments $m[k_1,...,k_N]$, for $k_1,...,k_N =1,2,...$. 

%The rooted spectral measures $\mu_i$ are the marginal distributions of $\mu$. By a direct calculation,
	%$$ \mu \big( \{ \lambda_\sigma : \sigma(i)=k \} \big) = \sum_{\substack{ \sigma \in S_N \\ \sigma(i)=k}}  \epsilon(\sigma) \prod_{j=1}^N p_{j \sigma(j)} = p_{ik} \sum_{\substack{ \sigma \in S_N \\ \sigma(i)=k}} \epsilon(\sigma) \prod_{j \neq i} p_{j \sigma(j)} = p_{ik} \operatorname{cof}(P)_{ki}. $$
%where $\operatorname{cof}(P)_{ki}$ is the $(k,i)$ cofactor of $P$. Since $P$ is orthogonal with determinant $1$, $\operatorname{cof}(P)_{ki}=p_{ik}$, yielding 
%$$ \mu \big( \{ \lambda_\sigma : \sigma(i)=k \} \big) = p_{ik}^2 = \mu_i(\lambda_k) \ , \ k=1,...,n.  $$	

%The marginal distributions of $\mu$ also have relatively simple expressions related to the determinant of minors of $P$. 

\begin{proposition} The marginal distributions of $\mu$ are the rooted spectral measures $\mu_1,...,\mu_N$. In particular, if all eigenvalues $\lambda_1,...,\lambda_N$ are distinct,
$$ \mathbb P(X_i = \lambda_k) = p_{ik}^2 \ , \ \forall i,k=1,...,N.  $$
%$$ \mu \big( \{ \lambda_\sigma : \sigma(u) = v \} \big) = \epsilon\big( \sigma_{uv} \big) \det\big(P_{uv} \big) \prod_{k=1}^K p_{s_k t_k}, $$
%$$ \mathbb Q \big( X_{s_1} = \lambda_{t_1},..., X_{s_K} = \lambda_{t_k} \big) = \epsilon\big( \sigma_{st} \big) \det\big(P_{st} \big) \prod_{k=1}^K p_{s_k t_k}, $$
%where $\sigma_{uv}$ is the only permutation such that $\sigma(u) = v$ and $\sigma(\overline u) = \overline v$.
\end{proposition}

\begin{proof} The $i$-th marginal distribution of $\mu$ is characterized by the moments $\mathbb E(X_i^k) = m[0,...,0,k,0,...,0]$ for $k=1,2,...$ placed at position $i$. By Lemma \ref{lem:gen_mom}, 
$$ \mathbb E \big(X_i^k \big) = \det \big(A[0,...,0,k,0,...,0] \big) = (A^k)_{ii}, $$
ending the proof.
\end{proof}

\noindent The explicit formula of the generalized moments provides a powerful tool to prove some interesting properties of the joint spectral measure. For instance, expressing the $(i,j)$-minor of $A^k$ in function of the moments of $X_i,X_j$, we obtain
$$ \operatorname{cov}\big(X_i^k,X_j^k\big) := \mathbb E \big( X_i^k X_j^k) - \mathbb E(X_i^k) \mathbb E(X_j^k) = - \big(A^k\big)_{ij}^2 \ , \ i \neq j \ , \ k \in \mathbb N. $$ 
This shows in particular that $X_i^k$ and $X_j^k$ are always non-positively correlated for $i \neq j$. When $A$ is the adjacency matrix of an undirected graph $G$ with $a_{ij} \in \{ 0,1\}$ for all $i,j=1,...,N$, $-\operatorname{cov}\big(X_i^k,X_j^k\big)$ gives the squared number of walks of length $k$ from $i$ to $j$. If we assume moreover that the graph contains no self-loop (i.e.~$A$ has zero diagonal), we can verify by direct calculation that the covariance matrix of $\mu$ is the Laplacian of the graph:
$$ \operatorname{var}(X) = \mathbb E \big( X  X^\top) - \mathbb E(X) \mathbb E(X)^\top = L,  $$
where $L_{ij} = - a_{ij} $ for $i \neq j$ and $L_{ii}$ is the degree of the vertex $i$ in $G$. \\

Another straightforward consequence of Lemma \ref{lem:gen_mom} concerns the minors of $A$, which are linked to the quasi-random variables $X_i$ by the identity $ \det \big( A_{uu} \big) = \mathbb E \big( \ \prod_{i \in u} X_i \big)$, where $u$ is a subset of $\{ 1,...,N \}$ and $A_{uu}$ the corresponding submatrix. The cycle decomposition of the determinant shows a relation between the cumulants of $X_i, i \in u$ and the simple cycles with vertex set $u$. Consider the set $\mathcal P(u)$ of all the partitions of $u$, that is, all the collections of subsets $\{\pi_1,...,\pi_k\}$ with $k \geq 1$ such that $\pi_i \cap \pi_j = \emptyset$ for all $i \neq j$ and $\cup_{j=1}^k \pi_j = u$. Let $c(u)$ denote the sum of all simple cycles on $G$ with vertex set $u$, we have the well-known equality
$$ \det \big( - A_{uu} \big) = \sum_{(\pi_1,...,\pi_k) \in \mathcal P(u)} (-1)^k c(\pi_1) \times  ... \times  c(\pi_k).  $$
On the other hand, the multivariate first-order cumulants $\kappa(\pi_j)$ associated to the variables $X_i, i \in \pi_j$ are defined in such a way that
$$ \mathbb E \Big( \ \prod_{i \in u} X_i \Big) = \sum_{(\pi_1,...,\pi_k) \in \mathcal P(u)} \kappa(\pi_1) \times ... \times \kappa(\pi_k). $$
%see for instance \cite{arizmendi2015relations}. 
It follows that $\kappa(u) = (-1)^{|u|-1} c(u)$ where $|u|$ is the size of $u$. \\

The identity equating a minor of $A$ to a generalized moment can be extended to any power $A^k, k \in \mathbb N$, e.g.~$ \det \big( (A^k)_{uu} \big) = \mathbb E \big( \prod_{i \in u} X_i^k \big)$, as another direct consequence of Lemma \ref{lem:gen_mom}. In fact, a similar result holds for any analytical transformations $f(A) = \sum_{k\geq 0} \gamma_k A^k$.

\begin{proposition}\label{prop:gen_mom_fun} Let $f: x \mapsto \sum_{k\geq 0} \gamma_k x^k$ be an analytical function with spectral radius $\rho > \max \{ | \lambda_j | ,j =1,...,N \}$ and $u \subseteq \{1,...,N\}$, 
$$\det\big( f(A)_{uu} \big) = \mathbb{E}\bigg( \prod_{i \in u} f(X_i) \bigg).$$
\end{proposition}

\begin{proof} Assume without loss of generality that $u=\{1,...,p\}$ with $1 \leq p \leq N$. By multilinearity of the determinant
$$ \det\big( f(A)_{uu} \big) = \det \bigg( \sum_{k\geq 0} \gamma_k (A^k)_{uu} \bigg) = \sum_{k_1,...,k_p\geq 0} \gamma_{k_1}...\gamma_{k_p} \det \big( A[k_1,...,k_p,0,...,0] \big). $$
Using Lemma \ref{lem:gen_mom}, we get 
$$ \det\big( f(A)_{uu} \big) =  \mathbb E \bigg( \sum_{k_1,...,k_p \geq 0}  \gamma_{k_1}...\gamma_{k_p}X_1^{k_1} ... X_p^{k_p}  \bigg) = \mathbb{E}\bigg( \prod_{i \in u} f(X_i) \bigg),$$
ending the proof.
\end{proof}

Interestingly, the same equality holds for the trace $\tr \big( f(A)_{uu} \big) = \mathbb E \big( \sum_{i \in u} f(X_i) \big)$, although it only pertains to the marginal distributions $\mu_i$. This particular coupling somewhat allows to obtain a similar property for the determinant. This shows moreover that the quasi-random variables $X_i, i \in u$ encapsulate the information on the spectrum of the submatrix $f(A)_{uu}$, for any analytical transformation of $A$. Specifically, the eigenvalues of $f(A)_{uu}$ are the roots of the characteristic polynomial 
$$ z \mapsto \mathbb{E}\bigg( \prod_{i \in u} \big( z - f(X_i) \big) \bigg) = \det \big( z I - f(A)_{uu} \big). $$

\begin{proposition}\label{prop:slater} Assume that the eigenvalues $\lambda_1,...,\lambda_N$ of $A$ are distinct. Let $u ,v$ be two subsets of $\{1,...,N \}$,
$$ \mathbb P \Big( \{X_{i} : i \in u \} = \{ \lambda_{j}: j \in v \} \Big) = \det \big( P_{uv} \big)^2,    $$
where $P_{uv}$ is the submatrix of $P$ with rows in $u$ and columns in $v$ (in any particular order).
\end{proposition}

\begin{proof} Since the ordering of the eigenvalues $\lambda_j$ is arbitrary, we may assume $u=v$ without loss of generality (note that the condition $\det(P)=1$ is not an issue here as it can be ensured without changing the order of the eigenvalues). We have
$$ \mathbb P \Big( \{X_{i} : i \in u \} =  \{ \lambda_{j}: j \in v \} \Big)= \sum_{\sigma: \sigma(u) = u} \epsilon(\sigma) \prod_{i=1}^N p_{i \sigma(i)}, $$
where the sum runs over all permutations $\sigma$ that stabilize $u$. Such permutations can be associated with the pair $(\sigma_u,\sigma_{\overline u})$ of permutations over $u$ and $\overline u = \{1,...,N \} \setminus u$ respectively, corresponding to the restrictions of $\sigma$ to $u$ and $\overline u$. Noticing that $\epsilon(\sigma) = \epsilon(\sigma_u) \epsilon(\sigma_{\overline u})$, we get
$$ \mathbb P \Big( \{X_{i} : i \in u \} =  \{ \lambda_{j}: j \in v \} \Big)= \sum_{\sigma_u \in S(u)} \epsilon(\sigma_u) \prod_{i \in u}^N p_{i \sigma_u(i)} \times \sum_{\sigma_{\overline u} \in S(\overline u)} \epsilon(\sigma_{\overline u}) \prod_{i \in \overline u}^N p_{i \sigma_{\overline u}(i)}, $$
where $S(u)$ and $S(\overline u)$ are the sets of permutations over $u$ and $\overline u$. Thus, 
$$ \mathbb P \Big( \{X_{i}: i \in u \} =  \{ \lambda_{j}: j \in u \} \Big) = \det\big(P_{uu}\big) \det\big(P_{\overline u \overline u} \big).  $$
We conclude by Jacobi's Identity (see e.g.~Equation $(11)$ in \cite{jacobi}): $\det\big(P_{\overline u \overline u} \big) = \det\big(P_{u u} \big) / \det (P)$.
\end{proof}

This last result can also be shown directly, together with a stronger property giving the multidimensional marginal distributions of $\mu$, see Section \ref{a:margin} in the Appendix.\\

In quantum mechanics, the last expression corresponds (up to a normalization constant) to the square of the wave function $\Psi_u(\lambda_v) := \det (P_{uv})$ known as Slater determinant \cite{slater}. In this context, the coupling $\mu$  provides a non-local interpretation of the quantum wave-function from a quasi-probability, where the square emerges naturally from marginal distributions.

\section{A multivariate central limit theorem for star graphs}\label{sec:3}

In this section, we generalize a result of Obata \cite{obatatcl} on the asymptotic behavior of the spectral measure in star graphs. Let us first recall the original result: let $G$ be a rooted graph (that is a graph given with a special vertex $o$ called the root) with adjacency matrix $A$, and $G^{(n)}$ the star product defined by taking $n$ copies of $G$ and merging all the roots. Let $\mu_o^{(n)}$ be the spectral measure of $G^{(n)}$ at the root $o$ and $d_o$ the degree of $o$ in $G$.

\begin{theorem}[Theorem 3.7 \cite{obatatcl}]\label{th:1} The (normalized) spectral measure $\mu_o^{(n)}$ converges weakly as $n \to \infty$ with
$$ \frac 1 {\sqrt n} \mu^{(n)}_o \Big( \frac{.}{\sqrt n} \Big) \underset{n \to \infty}{ -\!\!\!\rightharpoonup} \frac 1 2 \big( \delta_{-\sqrt{d_o}} + \delta_{\sqrt{d_o}} \big).$$
%$\mu_o^{(n)}$
%When $n \rightarrow +\infty$, the following convergence holds :
 %$$ \frac{X^{(n)}_o}{\sqrt{n}} \xrightarrow{\mathcal{D}} \sqrt{d}B,$$
 %where $d$ denotes the degree of the root $o$ in $G$, and $B$ is a Rademacher random variable. 
\end{theorem}

 The proof relies on showing the convergence of moments via the Sz\"ego-Jacobi sequence derived from the adjacency matrix of $G^{(n)}$. Obata's result can be stated equivalently as the convergence in distribution
$$ \frac{X^{(n)} }{\sqrt n} \xrightarrow[n \to \infty]{d} B \sqrt{d_0}, $$
where $X^{(n)}$ is a random variable with distribution $\mu_o^{(n)}$ and $B$ is a Rademacher random variable: $\mathbb P(B=1)=\mathbb P(B=-1)=1/2$. We show a multivariate version of this result, where $G^{(n)}$ is the the graph obtained by merging the $n$ copies of $G$ at $p \geq 1$ vertices, say $u_1,...,u_p$. For simplicity and without loss of generality, we assume that $u_1,...,u_p$ are the first $p$ vertices of $G$. As a result, the adjacency matrix $A^{(n)}$ of $G^{(n)}$ contains $p+n(N-p)$ rows and columns and can be decomposed by blocks as
$$ A^{(n)} = \left[ \begin{array}{cccc} 
A_{uu} & A_{u \overline u} & \dots &  A_{u \overline u} \\ 
A_{\overline u u } & A_{ \overline u \overline u} & 0 & 0 \\
\vdots & 0 & \ddots & 0 \\ 
A_{\overline u u } & 0 & 0 & A_{ \overline u \overline u} \end{array}
 \right]  $$
where $u = \{u_1,...,u_p\}$ and $\overline u = \{1,...,N\}\setminus u$. We consider the joint spectral measure $\mu^{(n)}_u$ of the first $p$ vertices in the graph $G^{(n)}$. For quasi-random variables, the convergence in distribution signifies the weak convergence of the signed measures.

\begin{theorem} Let $X^{(n)}=\big(X_1^{(n)},...,X_{p}^{(n)}\big)$ be a quasi-random vector with distribution $\mu^{(n)}_u$, 
 $$ \frac{ \big( X^{(n)}_1 ,...,X^{(n)}_p \big)}{\sqrt{n}} \xrightarrow[n \to \infty]{d}  \big(B_1\sqrt{Y_1}, \cdots,B_p\sqrt{Y_p}  \big),$$
where $Y=(Y_1,...,Y_p)$ is a quasi-random vector with distribution the joint spectral measure of $D = A_{u \overline u} A_{\overline u u}$ and $B_1,...,B_p$ are iid Rademacher random variables, independant of $Y$.
\end{theorem}

\begin{proof} Recall that the $X^{(n)}$ takes values in the permutations of the eigenvalues of $A^{(n)}$. Let $\| . \|$ and $\| . \|_F$ denote respectively the Euclidean and Frobenius norms, we have
\begin{equation}\label{proof:eq1} \Vert X^{(n)} \Vert^2 = \|A^{(n)} \|^2_F \leq n \|A\|_F^2.
\end{equation}
Thus, $\Vert X^{(n)} \Vert/\sqrt n \leq \|A\|_F$ almost-surely. To prove the weak convergence, it now suffices to prove the convergence of the moments
$$ m^{(n)}_u[k_1,...,k_p] := \mathbb E \bigg( \bigg(\frac{X^{(n)}_1}{\sqrt n}\bigg)^{k_1} ... \ \bigg(\frac{X^{(n)}_p}{\sqrt n}\bigg)^{k_p} \bigg) = \frac 1 {\sqrt n^{k_1+...+k_p}} \det \Big(A^{(n)}[k_1,...,k_p,0,...,0] \Big) $$
%  \int_{\mathbb R^p} \frac{x_1^{k_1}...x_p^{k_p}}{\sqrt n^{k_1+...+k_p}} \ d\mu_u^{(n)}(x_1,...,x_p). $$
where the last equality follows from Lemma \ref{lem:gen_mom}.
%$$m^{(n)}_u[k_1,...,k_p] = \frac 1 {\sqrt n^{k_1+...+k_p}} \det \Big(A^{(n)}[k_1,...,k_p,0,...,0] \Big),$$ 
%The generalized moments thus depend continuously of the $(u,u)$-blocks of $\big(A^{(n) }\big)^k$ for $k=1,2,...$. 
By Schur's complement formula applied to $I - z A^{(n)}$, the $(u,u)$-block of the resolvant satisfies
$$ R_u^{(n)}(z) := \Big(\big(I - z A^{(n)} \big)^{-1} \Big)_{uu} = \Big( I-zA_{uu}-n z^2A_{u \overline u} \big(I- zA_{\overline{u}\overline{u}}\big)^{-1}   A_{\overline{u}u} \Big)^{-1}.  $$
\begin{eqnarray*} R_u^{(n)} \Big(\frac{z}{\sqrt{n}} \Big) & = &  \Big( I-\frac{z}{\sqrt{n}}A_{uu}-z^2A_{u \overline u} \Big(I- \frac{z}{\sqrt{n}}A_{\overline{u}\overline{u}}\Big)^{-1}   A_{\overline{u}u} \Big)^{-1} \\
&   = & I + z \frac{\big(A^{(n)}\big)_{uu}}{\sqrt n} + z^2 \frac{\big(A^{(n) 2}\big)_{uu}}{\sqrt n^2} + ... \\
& \xrightarrow[n \to \infty]{} & \big(I-z^2 D\big)^{-1} = I + z^2 D + z^4 D^2 + ...
\end{eqnarray*}
where we recall $D = A_{u \overline u} A_{\overline u u}$. The sub-multiplicativity of the Frobenius norm combined with Eq.~\eqref{proof:eq1} gives $\big\Vert \big(A^{(n)}\big)^k_{uu} \big\Vert_F/\sqrt n^k \leq \|A\|_F^k$. Thus, the series has positive convergence radius and by continuity of the determinant, we deduce
$$ m^{(n)}_u[k_1,...,k_p] \underset{n \to \infty}{\longrightarrow} \left\{ \begin{array}{cl}  \det \Big( D \big[k_1/2,...,k_p/2 \big] \Big) & \text{if $k_1,...,k_p$ are even,} \\ 
0 & \text{otherwise.}\end{array} \right.    $$
For $(z_1,...,z_p)$ in a sufficiently small neighborhood of $0$, the moment generating function verifies
\begin{align*} \mathbb E \bigg(  \prod_{i =1}^p \frac{1}{1-z_i X^{(n)}_i/\sqrt n } \bigg) & = \sum_{k_1,...,k_p =0}^\infty z_1^{k_1}...z_p^{k_p} \ m^{(n)}_u[k_1,...,k_p] \\
& \underset{n \to \infty}{\longrightarrow} \sum_{k_1,...,k_p =0}^\infty z_1^{2k_1}...z_p^{2k_p} \ \det \big( D [k_1,...,k_p] \big) = \mathbb E \bigg(  \prod_{i =1}^p \frac{1}{1-z_i^2 Y_i} \bigg) 
\end{align*}
where $Y=(Y_1,...,Y_p)$ is a quasi-random vector with distribution the joint spectral measure of $D$. Finally, let $B_1,...,B_p$ be iid Rademacher variables independent of $Y$, we have by Fubini's theorem
$$ \mathbb E \bigg( \prod_{i =1}^p \frac{1}{1-z_i B_i \sqrt{Y_i}} \bigg) = \mathbb E \bigg( \prod_{i =1}^p \bigg[ \frac 1 2 \frac{1}{1-z_i \sqrt{Y_i}} + \frac 1 2 \frac{1}{1+z_i \sqrt{Y_i}} \bigg] \bigg)= \mathbb E \bigg(  \prod_{i =1}^p \frac{1}{1-z_i^2 Y_i} \bigg).  $$
Hence, $\big(X^{(n)}_1/\sqrt n,...,X^{(n)}_p/ \sqrt n \big) $ converges in distribution towards $(B_1  \sqrt{Y_1},...,B_p \sqrt{Y_p} \big)$.
\end{proof}

Remark that because $D= A_{u \overline u} A_{\overline u u}$ is positive definite, $Y_i$ is always non-negative almost-surely and $\sqrt{Y_i}$ is well-defined. The entry $D_{u_i u_j}$ gives the number of vertices in $\overline u$ that are adjacent to both $u_i$ and $u_j$. Obata's result corresponds to the particular case $u =\{o\}$ where $D$ is equal to the degree $d_o$ of the root $o$ in $G$ and the associated joint spectral measure is a Dirac mass at $d_o$.

\section{Combinatorial properties}

In this section, we investigate the combinatorial aspects of the joint spectral measure via its relations with paths enumeration on a graph $G$. Thus, $A$ is viewed here as the (possibly weighted) adjacency matrix of a graph $G$ with vertex set $\{1,...,N\}$. As usual when one is interested in enumerating paths in $G$, the entries $a_{ij}$ of $A$ may be thought of as formal variables whenever $(i,j)$ is an edge and $a_{ij} =0$ otherwise. A path of length $n$ from $i$ to $j$ in $G$ is a succession of $n$ contiguous edges $w= a_{i i_1} a_{i_1 i_2} ... a_{i_{n-1 j}}$. The length of a path $w$ is denoted by $\ell(w)$. By convention, the null path $1$ is a path of length $0$ from one vertex to itself.\\

A simple cycle is a closed path in the graph that does not visit the same vertex twice before its return to its starting vertex. When endowed with a specific partially commutative rule, product of cycles form well studied algebraic objects originally introduced as circuits by Cartier and Foata \cite{cartier1969} later  revisited as heaps of cycles by Viennot \cite{viennot1989heaps} or hikes \cite{giscard2017algebraic}. The free partially commutative monoid (or trace monoid) of hikes consists of all finite products of simple cycles $h=c_1...c_n$ (some cycles may be repeated), allowing to permute two consecutive cycles only if they have no vertex in common. For instance, $ab^2 = bab = b^2a$ if $a$ and $b$ are vertex disjoint in $G$ but all three terms all different if $a$ and $b$ share at least one vertex (and $a\neq b$). This somewhat peculiar structure is heavily related to the spectral properties of the graph. In particular, letting $\mathcal H$ denote the set of hikes in $G$, the zeta function of $\mathcal H$ (or characteristic function) is the determinant of the resolvant $R(z) =  (I - z A)^{-1}$
$$ \zeta(z) := \sum_{h \in \mathcal H} h z^{\ell(h)} = \det \big( R(z) \big) , $$
where $\ell(h)$ is the length of $h$, that is the added length of all cycles composing $h$. Equivalently, the (slightly modified) characteristic polynomial of $A$ is the Mobi\"us function of $\mathcal H$,
$$ M(z) := \frac{1}{\zeta(z)} = \det \big(I-zA \big), $$ 
whose expression in terms of products of vertex-disjoint cycles is well known. Due to the numerous similarities with their number theoretic counterparts, simple cycles have been described as the prime elements in $\mathcal H$. The representation of a hike as a product of simple cycles is its prime decomposition, which is unique up to permuting consecutive vertex-disjoint cycles \cite{giscard2017algebraic}. We say that $d \in \mathcal H$ is a right divisor of a hike $h$ if there exists $h' \in \mathcal H$ such that $h = h'd$. \\

While all closed paths are hikes, the reverse is clearly not true. Viennot remarked that closed paths can be characterized as hikes with a unique prime right divisor, called pyramids in the context of heaps of pieces (see Proposition 7 in \cite{viennot1989heaps}). However, when viewed as a hike, a closed path $w$ loses the information of its starting point. In fact, the vertices in the unique prime right divisor of $w$ are the possible starting points of a closed path. Thus, the number of closed paths that are associated with the same hike is equal to the length of the unique right prime divisor. This observation gave rise to the so-called hike von-Mangoldt function
$$ \Lambda(h) := \left\{ \begin{array}{cl} \ell(p) & \text{if $h$ has $p$ as its unique prime right divisor of $h$} \\ 0 & \text{otherwise,} \end{array} \right.  $$
whose associated generating function is given by the trace of the resolvant 
$$ \sum_{h \in \mathcal H} \Lambda(h) h z^{\ell(h)} =  \tr \big( R(z) \big) = \tr \big( I + zA + z^2 A^2 + ... \big).  $$
More details can be found in \cite{giscard2017algebraic}.

\begin{definition} An excursion on a proper subset $u$ of vertices in $G$ is a path that starts and ends in $u$ but does not visit $u$ in between. Formally, a path $w= a_{i i_1} a_{i_1 i_2} ... a_{i_{n-1 j}}$ is an excursion on $u$ in $G$ if $i,j \in u$ and $i_k \in \overline u, \forall k =1,...,n-1$. 
\end{definition}

Let $\mathcal E_{ij}(u)$ denote the set of excursions from $i$ to $j$ on $u$ in $G$ and $E_u(z)$ the matrix generating function of the excursions on $u$. Since an excursion on $u$ in $G$ can be decomposed uniquely as the concatenation of an edge from $u$ to $\overline u$ followed by a path (possibly empty) in $\overline u$ and a edge from $\overline u$ to $u$, one verifies easily that
$$ E_u(z) := \Big( \sum_{w \in \mathcal E_{ij}(u)} w \ z^{\ell(w)}\Big)_{i,j \in u}  = z A_{uu} + z^2 A_{u \overline u} \big(I - z A_{\overline u \overline u} \big)^{-1} A_{\overline u u}. $$
From this expression, it is apparent that $I-E_u(z)$ is in fact the Schur complement of the block $I - zA_{uu}$ in $I-zA$.
%$$ I - z A = \left[ \begin{array}{cc} I - z A_{uu} & -zA_{u \overline u} \\ -zA_{\overline u u} & I - z A_{\overline u \overline u} \end{array} \right]. $$
Hence, excursions emerge naturally from submatrices of the resolvant $R(z)=\big(I-zA\big)^{-1}$ by
\begin{equation}\label{eq:mat_exc} \big(I - E_u(z) \big)^{-1} = \big( R(z) \big)_{uu} =:R_u(z). \end{equation}
The associated minor of the resolvant, 
$$ r_u(z) := \det \big(I - E_u(z) \big)^{-1} = \det \big( R_u(z) \big) $$
is a generating function of a certain type a hikes, as we show in Proposition \ref{prop:es2} below. Minors of the resolvant can also be expressed as generalized moments of the joint spectral measure. A direct application of Proposition \ref{prop:gen_mom_fun} to $f(x) = 1/(1-zx)$,
$$ r_u(z) = \det \big( R_u(z) \big) = \mathbb E \Big( \prod_{i \in u} \frac{1}{1-z X_i} \Big),  $$
shows that $r_u(z) $ is the (homogenous) moment generating function of $(X_i)_{i \in u}$.  Remark that the series $r_u(z)$ is not the hike zeta function in the induced subgraph $G(u)$. The latter is simply given by
$$ \zeta_u(z) := \det \Big( \big(I-zA_{uu} \big)^{-1} \Big) \neq \det \Big(\big(I-zA \big)^{-1}_{uu} \Big).  $$
\begin{proposition}\label{prop:es2} The series $r_u(z)$ is the generating function of hikes whose right divisors all intersect $u$.
\end{proposition}

\begin{proof} Recall that $I - E_u(z)$ is the Schur complement of the block $I-zA_{uu}$ in $I-zA$, in particular
$$ \det \big( I - zA\big) = \det \big( I-z A_{\overline u \overline u} \big) \det \big(I - E_u(z) \big), $$
%$$ \det \big(I - E_u(z) \big)^{-1} = \frac{\det \big( I-z A_{\overline u \overline u} \big)}{\det \big( I - zA\big)} = \zeta(z). M_{\overline u}(z)   $$
or equivalently $r_u(z) = \zeta(z) / \zeta_{\overline u}(z)$. By Proposition 5 in \cite{viennot1989heaps}, multiplication of $\zeta(z)$ by the Mobi\"us function on $G(\overline u)$ cancels out all hikes with at least one right prime divisor in $G(\overline u)$. 
\end{proof}

% \tib{
% \textbf{Désolé, c'est un brouillon, et je sais pas si on met une remarque dessus, j'arrive pas à l'écrire}
% \begin{corollary}
%  Let $h_k$ be the complete homogeneous polynomial of order $k$. Identifying homogeneous terms in the previous sum, we get
%  $$\E[h_k(X_{i_1}, \cdots X_{i_n})]  = \sum ?.$$
%  \end{corollary}
% We can compare this result with the trivial case of elementary symmetric polynomial $e_k$ verifing
% 
% $$\E[ e_k(X_{i_1}, \cdots X_{i_n}) ] ) = \ind_{k\leq n }  \sum_{s \subset \{i_1, \cdots, i_n \}} \det(A_{ss}), $$
%   and with power sum symetric polynomial $p_k$, that verifies
%   $$\E[ p_k(X_{i_1}, \cdots X_{i_n}) ] ) = Tr(A^k_{ss}).$$
%  
% }

The link between $r_u(z)$ and the excursions comes from the fact that a hike has all its right divisors intersecting $u$ if and only if it can be decomposed as a product of excursions on $u$. Similarly, a closed path $w$ starting from a vertex in $u$ can be divided into a succession of excursions that eventually returns to its starting point. For the next result, let $\ell_u(h)$ denote the number of excursions on $u$ that compose a hike $h$, or equivalently the number of vertices in $u$ visited by $h$, counted with multiplicity. We define the function
$$ \Lambda_u(h) := \left\{ \begin{array}{cl} \ell_u(p) & \text{if $h$ has $p$ as its unique prime right divisor of $h$} \\ 0 & \text{otherwise} \end{array} \right. , \ h \in \mathcal H , $$
generated by the trace of the $(u,u)$-block of the resolvant: $\sum_{h \in \mathcal H} \Lambda_u(h) h z^{\ell(h)} = \tr \big ( R_u(z) \big)$, 
%We verify easily that the trace of $R_u(z)$ enumerates the 
%$$ \tr \big( R_u(z) \big) = \sum_{h \in \mathcal H} \Lambda_u(h) h z^{\ell(h)}  $$

%For a hike $h$, let $\Lambda_u(h)$ be the number of vertices in $u$ in the right prime divisor of $h$ is it is unique (i.e. if $h$ is a closed walk) and $\Lambda_u(h)=0$ otherwise. Clearly,
%$$ \sum_{h \in \mathcal H} \Lambda_u(h) h z^{\ell(h)} = \tr \big( R_u(z) \big) = \tr \big( I + z A_{uu} + z^2 (A^2)_{uu} + ... \big).  $$
\begin{proposition}\label{prop:es3} We have
$$\log \big(r_u(z) \big) = \sum_{h \in \mathcal H} \frac{\Lambda_u(h)}{\ell_u(h)} h z^{\ell(h)},$$
with the convention $\Lambda_u(h)/\ell_u(h)=0$ if $\ell_u(h)=0$. 
\end{proposition}

\begin{proof} From Eq.~\eqref{eq:mat_exc}, $ \log \big(r_u(z) \big) = - \log \big( \det \big( I - E_u(z) \big) \big) = \tr \big( - \log (I-E_u(z)) \big)$. By a Taylor expansion of $-\log(1-x)$ at $x=0$, we get
$$ \log \big(r_u(z) \big) = \tr \Big( E_u(z) + \frac{E_u(z)^2}{2} + \frac{E_u(z)^3}{3} + ... \Big).$$
We now identify each path in the series. The denominator $\ell_u(h)$ results from the diagonal entries of $E_u(z)^k$ which enumerate the closed paths that can be written as a product of $k$ excursions on $u$. The numerator $\Lambda_u(h)$ is due to the trace,  which counts the number of possible starting points of each path.
\end{proof}

Let us investigate the particular case $u = \{i\}$. An excursion on $\{i\}$ is a closed paths starting (and ending) at $i$ that do not visit $i$ in between. Since every closed paths from $i$ to $i$ can be decomposed uniquely as a product of excursions on $\{i\}$, we verify easily the relations
$$ R_{\{i\}}(z) = \sum_{w : i \to i} w z^{\ell(w)} = \frac{1}{1-E_{\{i\}}(z)} = 1+E_{\{i\}} + E_{\{i\}}^2 +... $$
Remark that a hike whose right divisors all intersect $\{i \}$ is a closed path, as two primes having a common vertex do not commute. Hence, Proposition \ref{prop:es3} is trivially true in this case with
$$ r_{\{ i \}} (z) = \det \big( R_{\{i\}} (z) \big) =  R_{\{i\}}(z). $$
Following \cite{boolean}, $E_{\{i\}}(z)$ corresponds to the $B$-transform associated to moment generating function $r_{\{ i \}} (z)$. Hence, the Boolean cumulants of $X_i$ enumerate the excursions on $\{i\}$. \\

In the case $u = \{ i \}$, $\Lambda_{\{i\}} (h) \in \{ 0,1 \}$ with the value $1$ if and only if $h$ is a closed path starting from $i$. Thus, Proposition \ref{prop:es3} yields
$$ \log \big( R_{\{i\}}(z) \big) = \sum_{h \in \mathcal H} \frac{\Lambda_{\{i\}}(h)}{\ell_{\{i\}}(h)} h z^{\ell(h)} = \sum_{w:i \to i} \frac{1}{\ell_{\{i\}}(w)} w z^{\ell(w)}. $$
In other words, the logarithm of the $i$-th diagonal entry of the resolvant is the generating series of closed paths starting from $i$, divided by their number of visits to $i$.

\section*{Appendix}

\subsection{Unicity of the joint spectral measure}\label{a:basis}

We give a direct proof that, although its definition involves a basis matrix $P$, the joint spectral measure $\mu$ does not depend on the choice of the eigendecomposition of $A$. We start with the following lemma.

\begin{lemma}\label{odot} Let $C=(c_{ij})_{i,j=1,...,N}$ be a block diagonal matrix with entries one in each block (and zero elsewhere) and $B=(b_{ij})_{i,j=1,...,N} $ a block diagonal matrix with support included in the support of $C$, i.e. such that $c_{ij} = 0 \Rightarrow b_{ij}=0$. Then, for all matrix $M=(m_{ij})_{i,j=1,...,N}$, 
$$ (MB) \odot C = (M \odot C)B.  $$
\end{lemma}
\begin{proof} Note that $b_{kj} c_{ij} = b_{kj}c_{ik}$ for all $i,j,k=1,...,N$ since the two sides of the equality are non-zero only if $i,j,k$ belong to the same block. Thus,
$$ \big( (M B) \odot C \big)_{ij} = \sum_{k=1}^N m_{ik} b_{kj} c_{ij} = \sum_{k=1}^N m_{ik} b_{kj} c_{ik} =  \big( (M \odot C) B \big)_{ij} $$
for all $i,j=1,...,N$.
\end{proof}
%Let $\overline \lambda_1,...,\overline \lambda_r$ be the distinct eigenvalues of $A$ and suppose that $\lambda = \big(\overline \lambda_1,...,\overline \lambda_1, ..... , \overline \lambda_r,...,\overline \lambda_r \big)$. 

\noindent Let $A = P\Lambda P^\top = Q \Lambda Q^\top$ be two eigendecompositions of $A$ with $\det(P) = \det(Q) = 1$. Without loss of generality, we may assume that equal eigenvalues are ordered consecutively in $\lambda$. By a play on the indices of the sum and product, the joint spectral measure can be written as
$$ \mu \big( \lambda_{\sigma} \big) = \sum_{\tau: \lambda_{\tau} = \lambda_{\sigma}} \epsilon(\tau)  \prod_{j=1}^N p_{j \tau(j)} = \sum_{\tau: \lambda_{\tau} = \lambda} \epsilon(\tau \circ \sigma)  \prod_{j=1}^N p_{j \tau \circ \sigma (j)} = \epsilon(\sigma) \sum_{\tau \in S_N} \epsilon(\tau)  \prod_{j=1}^N p_{\sigma^{-1}(j) \tau(j)} \mathds 1 \{ \lambda_j = \lambda_{\tau(j)} \}, $$
where $\mathds 1 \{ . \}$ is the indicator function. Define the block diagonal matrix $C = \big( \mathds 1 \{ \lambda_{i} = \lambda_j \} \big)_{i,j=1,...,N}$, we thus have
$$\mu (\lambda_\sigma) = \epsilon(\sigma) \det \big( (M_\sigma^\top P) \odot C \big) $$ 
where $M_\sigma$ is the permutation matrix associated to $\sigma$. Because the columns of $P$ and $Q$ are eigenvectors of the symmetric matrix $A$, the $i$-th column of $P$ is orthogonal to the $j$-th column of $Q$ whenever $\lambda_i \neq \lambda_j$. Therefore, $B:=Q^\top P$ and $C$ satisfy the conditions of Lemma \ref{odot}. Applying the lemma to $M = M_\sigma^\top Q$, we get
$$ \det \big( (M_\sigma^\top P) \odot C \big) = \det \big( (M_\sigma^\top Q B) \odot C \big) = \det \Big( \big( (M_\sigma^\top Q) \odot C \big) B \Big) = \det \big( (M_\sigma^\top Q) \odot C \big),  $$
in view of $\det(B) = \det( Q^\top P) = 1$. We conclude that $\mu$ does not depend on the basis matrix used in the eigendecomposition.

\subsection{Multivariate marginal distributions}\label{a:margin}

In this section, we give an expression of the multivariate marginal (quasi) distributions of $\mu$. We will assume for simplicity that $A$ has only simple eigenvalues. 

\begin{proposition} Let $s=\{s_1,...,s_k\}$ and $t=\{t_1,...,t_k\}$ be two subsets of vertices and $\sigma$ a permutation on $S_k$,
$$ \mathbb P \big( X_{s_1} = \lambda_{t_{\sigma(1)}} , ... , X_{s_k} = \lambda_{t_{\sigma(k)}} \big) = \epsilon(\sigma) \det \big( P_{st} \big) \prod_{j=1}^k p_{s_j t_{\sigma(j)}}. $$
%The multivariate marginal distributions of $\mu$ verifies, with $S,T$ being two subsets of $[1,N]$ ordered increasingly
% $$ \mu \big( \{ \lambda_\sigma : \sigma(u) = v \} \big) = \epsilon\big( \sigma_{uv} \big) \det\big(P_{uv} \big) \prod_{k=1}^K p_{s_k t_k}, $$
%$$ \pr\big( X_{s_1} = \lambda_{t_{\sigma_1}},..., X_{s_k} = \lambda_{t_{\sigma_k}} \big) = \sum_{\tau\in S_K: \lambda_{t_{\sigma(i)}} = \lambda_{t_{\tau(i)}}}\epsilon \big( \tau \big) \det\big(P_{ST} \big) \prod_{k=1}^K P_{s_k t_{\tau(k)}}, $$
% where $\sigma_{uv}$ is the only permutation such that $\sigma(u) = v$ and $\sigma(\overline u) = \overline v$.
\end{proposition}

\begin{proof} By definition
$$ \mathbb P \big( X_{s_1} = \lambda_{t_{\sigma(1)}} , ... , X_{s_k} = \lambda_{t_{\sigma(k)}} \big) = \sum_{\tau: \tau(s_j) = t_{\sigma(j)}} \epsilon(\tau) \prod_{j=1}^N p_{j \tau(j)}, $$
where the sum runs over all permutations $\tau$ such that $\tau(s_j) = t_{\sigma(j)}$ for $j=1,...,k$. Let $\overline s = \{\overline s_1,..., \overline s_{N-k} \} = \{1,...,N\} \setminus s$ and define $\overline t$ similarly. Let $\sigma_{st} \in S_N$ be the unique permutation such that $\sigma_{st}(s_j) = t_{\sigma(j)} $ and $\sigma_{st} (\overline s_j) = \overline t_j$ for all $j$. We verify easily that $\epsilon(\sigma_{st}) = \epsilon(\sigma)$. Letting $\tau'= \tau \circ \sigma_{st}$ in the equation above, we obtain
\begin{align*} \mathbb P \big( X_{s_1} = \lambda_{t_{\sigma(1)}} , ... , X_{s_k} = \lambda_{t_{\sigma(k)}} \big) & =  \prod_{j=1}^k p_{s_j t_{\sigma(j)}} \sum_{\tau': \tau'(t_j) = t_{j}} \epsilon(\sigma_{st}) \epsilon(\tau') \prod_{j=1}^{N-k} p_{\overline s_j \tau'(\overline t_j)}\\
& = \epsilon(\sigma)  \prod_{j=1}^k p_{s_j t_{\sigma(j)}} \det \big( P_{\overline s \overline t} \big).
\end{align*}
By Jacobi's Identity (Equation $(11)$ in \cite{jacobi}), $\det \big( P_{\overline s \overline t} \big) = \det \big( P_{s t} \big)/ \det(P)=\det \big( P_{s t} \big)$, which concludes the proof.
\end{proof}
%The proof can be easily adapted to the case of multiple eigenvalues although the notations are more tedious.

\bibliographystyle{unsrt}
\bibliography{biblio}

\end{document}